\documentclass[12pt]{article}
\usepackage[T1]{fontenc}

\usepackage{amssymb, amscd, amsmath, amsfonts, amsthm}
\usepackage{latexsym}
\usepackage{extarrows}

\usepackage{graphicx,subfigure,url}
\usepackage{epstopdf}

\usepackage{color}
\usepackage{tabu}

\usepackage{tikz}

\usepackage{mathtools}
\usepackage{enumitem}
\usepackage{pifont}
\usepackage{bbm}
\usepackage{hyperref}

\numberwithin{equation}{section}
\newtheorem{problem}{Problem}
\newtheorem{theo}[equation]{Theorem}
\newtheorem{remark}[problem]{Remark}

\newtheorem{defin}[equation]{Definition}
\newtheorem{prop}[equation]{Proposition}
\newtheorem{cor}[equation]{Corollary}
\newtheorem{lema}[equation]{Lemma}

\usepackage{xcolor}\definecolor{myblue}{RGB}{0,0,255}



\begin{document}

\date{}
\title{Indecomposability of the median hypersimplex and \\  polytopality of the hemi-icosahedral Bier sphere} 
\author{
Filip D. Jevti\'{c} \\{\small Mathematical Institute}\\[-2mm] {\small SASA,  Belgrade}
\and
Marinko \v{Z}. Timotijevi\'{c}\\ {\small Faculty of Science}\\[-2mm] {\small University of Kragujevac}
\and
Rade T. \v Zivaljevi\' c\\ {\small Mathematical Institute}\\[-2mm] {\small SASA, Belgrade}
}
\maketitle

\begin{abstract}
We prove that the median hypersimplex $\Delta_{2k,k}$ is Minkowski indecomposable, i.e.\ it cannot be expressed as a non-trivial Minkowski sum $\Delta_{2k,k} = P+Q$, where $P\neq \lambda\Delta_{2k,k}\neq Q$. Since $\Delta_{2k,k}$ is a deformed permutahedron, we obtain as a corollary that $\Delta_{2k,k}$ represents a ray in the submodular cone (the deformation cone of the permutahedron). Building on the previously developed geometric methods and extensive computer search, we exhibit a twelve vertex, $4$-dimensional  polytopal realisation of the Bier sphere of the hemi-icosahedron, the vertex minimal triangulation of the real projective plane.
\end{abstract}

\section{Introduction}

The \emph{wall-crossing relations}, more explicitly the equalities (i) and inequalities (ii), listed in Proposition \ref{prop:def-cone}, have recently found numerous applications in the theory of convex polytopes. 
These applications fall into two distinct, but closely related classes of problems: 
\begin{itemize}
    \item[(a)] the theory of deformations (deformation cones) of convex polytopes;
    \item[(b)] the existence of polytopal realizations of simplicial (polyhedral) spheres.
\end{itemize}
Our first main result (Theorem \ref{thm:main-1}) belongs to the first class. It claims  that the deformation cone of the median hypersimplex $\Delta_{2k,k}$ is one-dimensional, which implies that it 
represents a ray in the deformation cone of the permutahedron. Our second main result (Theorem \ref{thm:iko}) is in the second class. It says that the Bier sphere of the hemi-icosahedron admits a polytopal realization in $\mathbb{R}^5$. 

Bier spheres, see the Appendix (Section \ref{sec:app}), 
play an important role in both of these results. In the first they provide a simplicial refinement of the normal cone of $\Delta_{2k,k}$, needed for the application of Proposition \ref{prop:def-cone}, while in the second they are the central objects in an ongoing project of studying polytopality of combinatorial spheres.

\section{Deformation cones of polytopes}

Let $P\subset \mathbb{R}^d$ be a convex polytope and let $\mathcal{N}(P)$ be the corresponding normal fan. 
The \emph{deformation cone} $Def(P)$ is the convex cone parameterizing all \emph{admissible deformations} of $P$. More precisely, a convex polytope $Q\subset \mathbb{R}^d$
is an admissible deformation of $P$ (or simply a deformation of $P$) if the normal fan of $\mathcal{N}(Q)$ of $Q$ is a coarsening of $\mathcal{N}(P)$. 
The interior of $Def(P)$ is known \cite{mcmullen-1973} as the \emph{type cone} of $P$; it consists of all polytopes whose normal fan is exactly $\mathcal{N}(P) $. 

\medskip
As an immediate consequence of the definition we obtain that $Def(P)$ is closed under translations, $Q\mapsto Q+v$, dilations, $Q\mapsto \lambda Q\, (\lambda>0)$, and
Minkowski sums, $A, B \in Def(P) \Rightarrow A+B\in Def(P)$.

\medskip
The deformation cone is indeed a convex cone (in an appropriate parameter space) which is visible from the following list of statements equivalent to $Q\in Def(P)$.

\begin{enumerate}
  \item[{\rm (1)}]  $Q$ is obtained from $P$ by parallel displacement of its facets, without passing through a vertex \cite{postnikov-2009,postnikov-2008,castillo_liu};  
 \item[{\rm (2)}]  $Q$ is obtained from $P$ by perturbing the vertices so that the directions of all edges are preserved \cite{postnikov-2009,postnikov-2008};
 \item[{\rm (3)}]  $Q$ is a polytope whose support functionial is a convex piecewise
linear continuous function supported on the normal fan of $P$ \cite[Section 6.1]{cox-2011}\cite[Section 9.5]{loera-2010};
\item[{\rm (4)}] $Q$ is a Minkowski summand of a dilate of P \cite{meyer-1974,shepard-1963}. 
\end{enumerate}

The lineality subspace $\mathcal{L}(P)$ of $Def(P)$, that is the linear subspace of all apices of $Def(P)$, 
is $d$-dimensional (it is easily identified as $\mathcal{L}(P) = \{P+v \mid v\in \mathbb{R}^d\}$). The quotient 
\begin{equation}\label{eq:ess}
  Def_{ess}(P) := Def(P)/\mathcal{L}(P)
\end{equation}
is a pointed cone called the \emph{essential deformation cone} of $P$. 
By a slight abuse of language we often call (\ref{eq:ess}) the deformation cone of $P$. In particular the rays of $Def(P)$ are the rays of (\ref{eq:ess}),
$Def(P)$ is simplicial if (\ref{eq:ess}) is simplicial, etc.

\begin{prop}\label{prop:rays} {\rm (\cite{PPP22})}
A polytope $P$ has a one dimensional (essential) deformation cone if and only if it is Minkowski indecomposable. In particular 
the rays of $Def_{ess}(P)$ are spanned by the Minkowski indecomposable deformations of $P$.
\end{prop} 
Recall that a polytope $P$ is \emph{Minkowski indecomposable} \cite{mcmullen-1987,meyer-1974,PPP22,yost-2016,shepard-1963} if for each representation $\lambda P = Q + R, \, (\lambda > 0)$, both $Q$ and $R$ are homothetic translates of $P$.

\subsection{Deformed permutahedra and the rays of the submoduar cone}

The permutahedron $P_{n+1}$, defined as the convex hull of all permutations of the vector $(1, 2, \dots, n, n+1) \in \mathbb{R}^{n+1}$, has been 
one of the most studied polytopes in geometric and algebraic combinatorics. 

Deformed permutahedra were originally introduced by  Edmonds in 1970, under the name of \emph{polymatroids}, albeit his motivation was quite different.  In his influential paper \cite{edmonds1970} he introduced polymatroids as a polyhedral generalization of matroids, with far reaching consequences in combinatorial optimization.

They were rediscovered in 2009 under the name of generalized permutahedra by Postnikov \cite{postnikov-2009}. Since than this family of polytopes appeared naturally and found applications in several areas of mathematics, such as algebraic combinatorics \cite{aguiar2017,ardila2010,postnikov-2008}, optimization \cite{fujishige2005}, game theory \cite{danilov}, statistics \cite{morton2009},  mathematical economics \cite{joswig2021}, etc. 

\medskip
It is known that the set of deformed permutahedra (the deformation cone of $P_{n+1}$) can be parametrized by the cone of \emph{submodular functions} \cite{edmonds1970,postnikov-2009}.
However, faces of the submodular cone are far from being well understood, in particular determining the rays of $Def_{ess}(P_{n+1})$ remains an open problem since the 1970s \cite{edmonds1970}.

\medskip
One of our main results in this paper (Theorem \ref{thm:main-1}) identifies the \emph{median hypersimplex}, as the generator of a $S_{n+1}$-invariant ray in $Def_{ess}(P_{n+1})$.

\subsection{Minkowski decomposition of a permutahedron into hypersimplices}

\begin{defin}\label{def:hypersimplex}
  A hypersimplex $\Delta_{n,r}$ with parameters $n$ and $r$ is  defined as the convex hull of all $n$-dimensional vectors, vertices of the $n$-dimensional cube $[0,1]^n$, which belong to the hyperplane $x_1+\dots+ x_n = r$. Alternatively $\Delta_{n,r} = {\rm Newton}(e_r)$ can be described as the Newton polytope of the elementary symmetric function $e_r$ of degree $r$ in $n$ variables.
\end{defin}

Let $P_{n+1}(x) = P_{n+1}(x_1,x_2,\dots, x_{n+1}) \subset \mathbb{R}^{n+1}/ \mathbb{R}\cdot (1,1,\dots, 1)$ be the  permutahedron
\[
 P_{n+1}(x) = {\rm Conv}\{(x_{\pi_1}, x_{\pi_2},\dots, x_{\pi_{n+1}}) \mid \pi\in S_{n+1}\} \, .
\]

The following proposition is formulated without proof in \cite[Section 16]{postnikov-2009}.

\begin{prop} If $x_1 \geqslant x_2 \geqslant \dots \geqslant x_{n+1}$ then
there is a Minkowski sum decomposition
\begin{equation}\label{eqn:Mink}
P_{n+1}(x) =  (x_1-x_2)\Delta_{n+1,1} + (x_2-x_3)\Delta_{n+1,2}+\dots+ (x_n-x_{n+1})\Delta_{n+1,n} \, .
\end{equation}
\end{prop}

\medskip\noindent
{\bf Proof:}  Assume without loss of generality that $x_1 > x_2 > \dots > x_{n+1} > 0$. (The general case is obtained by passing to the limit.) Recall the identity
\[
x = x_1e_1+x_2e_2+ \dots+ x_{n+1}e_{n+1} = (x_1-x_2)E_1 + (x_2-x_3)E_2 +\dots+ (x_n-x_{n+1})E_n
\]
where $E_i:= e_1+e_2+\dots+ e_i $ and $E_0 = E_{n+1} = 0$. Similarly for each permutation $\pi\in S_{n+1}$ 
\begin{equation}\label{eqn:pi-identity}
x ^\pi:= \sum_{i=1}^{n+1} x_{\pi(i)}e_i = \sum_{j=1}^{n+1} x_{j}e_{\sigma(j)} = \sum_{j=1}^{n+1} (x_j-x_{j+1}) E_j^\sigma
\end{equation}
where $\sigma := \pi^{-1}$ and $ E_j^\sigma:= e_{\sigma(1)}+\dots+ e_{\sigma(j)}$.

\medskip
If $y\in  \mathbb{R}^{n+1}/ \mathbb{R}\cdot (1,\dots, 1)$ is a vector with coordinates   $y_{\pi(1)}> y_{\pi(2)} > \dots > y_{\pi(n+1)}$ then the extremal vertex of $P_{n+1}$ in the direction of $y$ is $x^\pi$ while
the right-hand side of (\ref{eqn:pi-identity}) is the extremal vertex of the right-hand side of (\ref{eqn:Mink}).  \qed

\medskip
The following special case of (\ref{eqn:Mink}) implies that all hypersimplices are deformed permutahedra, i.e.\ they belong to the deformation cone of the regular permutahedron $P_{n+1}$,

\begin{equation}\label{eqn:Mink-spec}
P_{n+1} = P_{n+1}(1,2,\dots, n+1) = \Delta_{n+1,1} + \Delta_{n+1,2}+\dots+ \Delta_{n+1,n} \, .
\end{equation}

\section{Median hypersimplex is indecomposable}

 We say that a fan $\mathcal{F}$ is supported by a set $S$ of vectors if every cone in $\mathcal{F}$ is spanned by a subset of $S$.
In other words $S$ has a representative in each ray of $\mathcal{F}$.
\begin{prop} {\rm (\cite[Proposition 2.1.]{PPP22})} \label{prop:def-cone}
Let $P\subset \mathbb{R}^d$ be a convex polytope whose normal fan $\mathcal{F}$ is refined by a simplicial fan $\mathcal{G}$ supported by $S$. Then the deformation cone $Def(P)$ of $P$ is the set of polytopes
\[
       P_h :=   \{x\in \mathbb{R}^d \mid \langle s, x \rangle \leqslant h_s \, \mbox{ for all } s\in S\}
\]
where the \emph{deforming vector} $h = (h_s)_{s\in S}$ satisfies the following conditions:

\begin{enumerate}
  \item[{\rm (i)}]  $\sum_{s\in R\cup R'}  \alpha_{R,R'}(s) h_s = 0$\, for each pair $Cone(R)$ and $Cone(R')$ of adjacent maximal cones in $\mathcal{G}$, belonging to \emph{the same} maximal cone in $\mathcal{F}$,
  \item[{\rm (ii)}] the inequalities $\sum_{s\in R\cup R'}  \alpha_{R,R'}(s) h_s \geqslant 0$\, for each pair $Cone(R)$ and $Cone(R')$ of adjacent maximal cones in $\mathcal{G}$, belonging to \emph{distinct} maximal cones in $\mathcal{F}$,
\end{enumerate}
where $\sum_{s\in R\cup R'}  \alpha_{R,R'}(s) s = 0$ is the unique linear dependence with $ \alpha_{R,R'}(r)+ \alpha_{R,R'}(r') =2 $ such that $R\setminus \{r\} = R'\setminus \{r'\}$.
\end{prop}

\begin{remark}\label{rem:wall}
  The relations (i) (respectively (ii)) in Proposition \ref{prop:def-cone} are frequently referred to as the \emph{wall crossing equalities} (respectively the \emph{wall crossing inequalities}).
\end{remark}

\subsection{Normal fan of the median hypersimplex}


Following Conway and Sloane \cite{CS91}, the \emph{diplo-simplex}\footnote{Unaware of \cite{CS91}, in \cite{jevtic_bier_2022} we called diplo-simplex the Van Kampen-Flores polytope.} is the convex polytope $\Omega_n = {\rm Conv}(\Delta\cup\nabla)$, where $\Delta = \Delta_u = \rm{Conv}(\{u_i\})_{i=1}^n$ is a regular simplex centered at the origin, and  $\Delta^\circ = -\Delta =: \nabla$ is the polar dual of $\Delta$.

Recall that  the vertices of $\Delta$ form a \emph{circuit} in $\mathbb{R}^{n-1}$ in the sense that
the linear map
\begin{equation}\label{eqn:circuit}
\mathbb{R}^n \stackrel{\Lambda}{\longrightarrow} \mathbb{R}^{n-1}, \,  \lambda = (\lambda_1, \dots, \lambda_n) \mapsto \Lambda(\lambda) := \lambda_1u_1+\dots+\lambda_nu_n
\end{equation}
is an epimorphism with the kernel generated by $\mathbbm{1} = (1,1,\dots, 1)\in \mathbb{R}^n$, meaning that
\begin{equation}\label{eq:rel}
 u_1+u_2+\dots + u_n = 0
\end{equation}
is, up to scaling, the only linear dependence of vertices of $\Delta_u$.

 \begin{prop}{\rm (\cite[Proposition 9]{jevtic_bier_2022})}\label{prop:facial}
 The vertex set of the polytope $\Omega_n$ is the set
 \begin{equation}\label{eq:vertices}
  Vert(\Omega_n) = \{u_1,u_2,\dots, u_n, -u_1, -u_2, \dots, -u_n\}\, .
 \end{equation}
 If a subset $\{u_i\}_{i \in I} \cup \{-u_j\}_{j\in J} \subset Vert(\Omega_n)$ corresponds to a proper face of $\Omega_n$ then
\begin{align}\label{eqn:facial-Omega}
    I\cap J = \emptyset \quad \mbox{and} \quad |I|,|J|\leq\frac{n}{2}.
\end{align}
Conversely, this condition is also sufficient if $n$ is an odd number. If $n$ is even, then a pair $(I,J)$ corresponds to a proper face of $\Omega_n$ if in addition to
\rm{(\ref{eqn:facial-Omega})} either (a) $\vert I\vert = \vert J\vert = \frac{n}{2}$, or (b) both $\vert I\vert$ and $\vert J\vert$ are strictly less than $\frac{n}{2}$.
 \end{prop}

The following theorem identifies the polar dual $\Omega_{2k}^\circ = \Delta\cap \nabla$ of the diplo-simplex $\Omega_{2k}$ as the \emph{median hypersimplex} $\Delta_{2k,k}$.

\begin{theo}{\rm (\cite[Theorem 14]{jevtic_bier_2022})}\label{thm:R_n=hypersimplex}
  If $n=2k$ is even then $\Omega_{2k}^\circ = \Delta\cap \nabla$ is affine isomorphic to the hypersimplex $\Delta_{2k,k}$.  If $n=2k+1$ then $\Omega_n^\circ$ is affine isomorphic to the convex hull
  \begin{equation}\label{eqn:hull}
  \begin{aligned}
    \Omega_{2k+1}^\circ \cong {\rm Conv}\lbrace\lambda \in [0,1]^{2k+1}  \mid  (\forall i)\, \lambda_i\in\{0,1/2,1\} \mbox{ {and} } \vert Z(\lambda) \vert = \vert W(\lambda) \vert = k
     \rbrace
\end{aligned}
  \end{equation}
 where   $Z(\lambda) = \{j \mid \lambda_j = 0\}$ and $W(\lambda)= \{j \mid \lambda_j = 1\}$.
\end{theo}

\begin{cor}\label{cor:NR}
If a polytope $P$ has the origin in its interior, then the normal fan $\mathcal{N}(P)$ corresponds to the \emph{radial fan} $\mathcal{R}(P^\circ)$  of the polar polytope $P^\circ$, constructed by taking the cone over each face of $P^\circ$. From here and Theorem \ref{thm:R_n=hypersimplex} we immediately deduce that
\[
         \mathcal{N}(\Delta_{2k,k}) = \mathcal{R}(\Omega_{2k}) \, .
\]
\end{cor}

\subsection{First main result}\label{sec:main-1}

Recall that a polytope $P$ is \emph{indecomposable} (more precisely Minkowski indecomposable) if for each representation $\lambda P = Q + R, \, (\lambda > 0)$, both $Q$ and $R$ are homothetic translates of $P$.

\begin{theo}\label{thm:main-1}
  Median hypersimplex $\Delta_{2k,k}$ is indecomposable.
\end{theo}

\begin{proof}
  It is well-known that a polytope $P$ is indecomposable if and only if the \emph{essential deformation cone} $Def_{ess}(P) = Def(P)/\mathcal{L}(P)$ is one-dimensional, where $\mathcal{L}(P)$ is the linear subspace of all apices of $Def(P)$ (corresponding to translations of $P$).

  \medskip
 We apply Proposition \ref{prop:def-cone} to the radial fan  $\mathcal{F}:=\mathcal{R}(\Omega_{2k})$ which, by Corollary \ref{cor:NR}, coincides with the normal   fan
  $\mathcal{N}(\Delta_{2k,k})$ of the median hypersimplex  $\Delta_{2k,k}$.

  \medskip
  In order to apply Proposition \ref{prop:def-cone}, we need to refine $\mathcal{F}$ by a simplicial fan $\mathcal{G}$. It is convenient that \emph{Bier spheres} provide a plethora of simplicial refinements of $\mathcal{F}$, see the Appendix, especially Corollary \ref{cor:VK-F}. After some preliminary steps, the corresponding wall-crossing equalities (see Remark \ref{rem:wall}) are listed in Lemma \ref{lema:wall-eq}.

\medskip
The (essentially) unique relation (\ref{eq:rel}) among the vertices of $\Delta_u$, yields the following relations among the vertices of the polytope $\Omega_n$, one for each \emph{balanced partition} $[n] = [2k] = S\uplus T$, $\vert S\vert = \vert T\vert = k$,
\[
        u_S := \sum_{i\in S} u_i =  \sum_{k\in T} u_{\bar{k}} =: u_{\overline{T}} \quad (\mbox{where } u_{\bar{k}} := -u_k)\, .
\]
In this case the corresponding relations (i) (Proposition \ref{prop:def-cone}) obtain a particularly simple form
\begin{equation}\label{eqn:ST}
       x_S := \sum_{i\in S} x_i =  \sum_{k\in T} y_k =: y_{T}
\end{equation}
 where, for convenience, we write  $x_i$ instead of $h_i$ and $y_k$ instead of $h_{\bar{k}}$.

\medskip
We want to show that the relation (\ref{eqn:ST}) is valid for each $S\subset [2k]$ of cardinality $k$ and $T=S^c$, or more precisely that it arises as one of linear relations (i) (Proposition \ref{prop:def-cone}) for an appropriate simplicial refinement  $\mathcal{G}$ of $\mathcal{F}$. Thanks to Proposition \ref{prop:facial}, Proposition \ref{prop:max-vol} and Corollary \ref{cor:VK-F}, it is sufficient to prove the following.

\begin{lema}\label{lema:wall-eq}
  Let $K$ be a simplicial complex satisfying the condition (\ref{eqn:even}) such that $S\notin K$. 
  Then $x_S = y_T$ is one of the wall-crossing equalities corresponding to the radial fan $\mathcal{R}(Bier(K))$ of the Bier sphere $Bier(K)$.
\end{lema}
For example we can choose
  \[
  K = \binom{[2k]}{\leq k-1} \, .
  \]

Indeed, $S\notin K$ implies $T\in K^\circ$. Moreover, $S\setminus \{i\}\in K$ for each $i\in S$. Let $\{i_1, i_2\}$ be two distinct elements of $S$ and
let $S_1 := S\setminus \{i_1\}$ and  $S_2 := S\setminus \{i_2\}$ be the corresponding faces of $S$. Then $(S_1,T;\{i_1\})$ and $(S_2,T;\{i_2\})$ are distinct facets
of $Bier(K)$ which share a common ridge (codimension-one face) $(S\setminus \{i_1, i_2\},T;\{i_1, i_2\})$  of $Bier(K)$. They all belong to the same facet of the diplo-simplex $\Omega_n$, corresponding to the pair $(S,T)$  (see Proposition \ref{prop:facial}) and the relation $x_S = y_T$ follows. \qed

\medskip
Side by side with $x_S = y_T$ there is a dual relation $x_T = y_S$, which together imply that
$$
x_{[n]} = x_S+x_T = y_S+y_T = y_{[n]} \, \mbox{ and }\, z_S:= x_S + y_S = x_{[n]} \, \mbox{ for each }  S\in \binom{[2k]}{k} \, .
$$
Summarizing, the linear span  ${Lin}Def(\Delta_{2k,k})$ of  $Def(\Delta_{2k,k})$, as determined by relations (i) (Proposition \ref{prop:def-cone}), is described by equations as follows,
\begin{equation}\label{eq:linear}
{Lin}Def(\Delta_{2k,k}) = \lbrace(x,y)\in \mathbb{R}^{2n} \mid (\forall S\in \binom{[2k]}{k} \, x_S + y_S = x_{[n]} = y_{[n]} \rbrace \, .
\end{equation}
Let $\mathcal{L}(\Delta_{2k,k})\subseteq {Lin}Def(\Delta_{2k,k})$ be the linear subspace of apices of $Def(\Delta_{2k,k})$. Let $P_h\in Def(\Delta_{2k,k})$ be the polytope with the deformation vector $h = (x,y)$, which means that $t\in P_h$ if and only if
\[
    (\forall i\in [n]) \quad  \langle u_i, t\rangle \leq x_i  \mbox{ and } \langle -u_i, t\rangle \leq y_i \, .
\]
The defining inequalities for the translated polytope $P_h+ v$ are
\begin{equation}\label{eq:translate}
(\forall i\in [n]) \,\,  \langle u_i, t+v\rangle \leq x_i+\langle u_i, v\rangle   \mbox{ and } \langle -u_i, t+v\rangle \leq y_i - \langle u_i, v\rangle\, .
\end{equation}
By choosing an appropriate translation vector $v$ we can satisfy additional \emph{normalizing conditions} on $x$ and $y$.
\begin{lema}\label{lema:normaliz}
The linearization of the \emph{essential deformation cone} 
\begin{equation*}
Def_{ess}(\Delta_{2k,k}) = Def(\Delta_{2k,k})/\mathcal{L}(\Delta_{2k,k})    
\end{equation*}
is the cone
$$
{Lin}Def(\Delta_{2k,k})/\mathcal{L}(\Delta_{2k,k})
$$
which is obtained if in addition to equalities (\ref{eq:linear}) we add the normalizing conditions,
\[
   (\forall i\in [n]) \,\, x_i = y_i \, .
\]
\end{lema}
Indeed, in light of the equality $x_{[n]} = y_{[n]}$ it is sufficient to find $v$ such that $x_i = y_i$ is satisfied for each $i=1,\dots, n-1$. This is achieved as follows.
Let $(v_j)_{j=1}^{n-1}$ be the basis in $\mathbb{R}^{n-1}$ dual to the basis $(u_j)_{j=1}^{n-1}$, i.e.\ the basis uniquely described by the condition $\langle u_i, v_j \rangle = \delta_{i,j}$. Than for each $j\in [n-1]$ there is a unique $\lambda_j$ such that the modification vector $\lambda_jv_j$ (applied in (\ref{eq:translate})) guarantees that the condition $x_j = y_j$ is satisfied and the unique desired translation vector is $v := \sum_{j=1}^{n-1} \lambda_jv_j$.

\medskip The lemma implies that the essential deformation cone $Def_{ess}(\Delta_{2k,k})$ is one-dimensional, since $x_{[n]} = y_{[n]}$ is the only variable parameter, which completes the proof of the theorem.
\end{proof}

\section{Polytopal Bier spheres of non-threshold complexes}\label{sec:Marinko}

In an earlier publication \cite{TZJ24}, we described an algorithm for finding an explicit polytopal realization of a given Bier sphere $Bier(K)$, by a sequence of successive modifications,
$$L=K_{0}\subset K_1\subset K_2\subset \cdots \subset K_k=K$$
and the corresponding sequence of Bier sphere approximations,
$$Bier(L)=Bier(K_0), Bier(K_1),\ldots, Bier(K_k)=Bier(K) \, ,$$
where each sphere $Bier(K_i)$ is obtained from $Bier(K_{i-1})$ by a local re-triangulation of $Bier(K_{i-1})$.\medskip

The initial step is the so called \emph{canonical polytopal realization} (see the Appendix, Section \ref{sec:app}) of the Bier sphere $Bier(L)$, where $L$  is a threshold complex $L$, chosen to be as close to $K$ as possible.\medskip

\medskip 
This algorithm was successfully applied to all Bier spheres with at most 11 vertices, as summarized by the following theorem.

\begin{theo}\label{thm:Bier-10-vertices}{\rm (\cite{TZJ24})}
All Bier spheres with up to eleven vertices are polytopal, in particular this holds for all $3$-dimensional Bier spheres. For illustration, there are $88$ non-threshold complexes on $5$ vertices and $48$ corresponding non-isomorphic Bier spheres. 
\end{theo}
The reader can find in \cite{TZJ24} additional information, including the links to explicit convex realizations of all  spheres with 10 and 11 vertices.

\subsection{Hemi-icosahedral Bier sphere}\label{sec:hemi}

Due to the complexity of calculation, Bier spheres with twelve vertices or more were not initially tractable by the algorithm described above.
With some improvements of the algorithm and by using more advanced computing resources we were able to tackle individual Bier spheres with 12 vertices.

A natural choice was the Bier sphere $Bier(\mathbb{I}_6)$ of the \emph{hemi-icosahedron} $\mathbb{I}_6$, the minimal, 6-vertex triangulation of the real projective plane $\mathbb{R}P^2$.  As depicted in Figure \ref{fig:hemi}, this complex arises as the $\mathbb{Z}_2$-quotient of the icosahedral sphere. It is not difficult to check that $\mathbb{I}_6$ is not a threshold complex so, in light of Theorem \ref{thm:Bier-polytope}, $Bier(\mathbb{I}_6)$ was our first candidate for a non-polytopal Bier sphere.

Contrary to our expectations, the sphere $Bier(\mathbb{I}_6)$ turned out to be polytopal and this is the second main result of our paper (Theorem \ref{thm:iko}).

\begin{figure}[htb]
\centering\vspace{+0cm}
\includegraphics[scale=0.55]{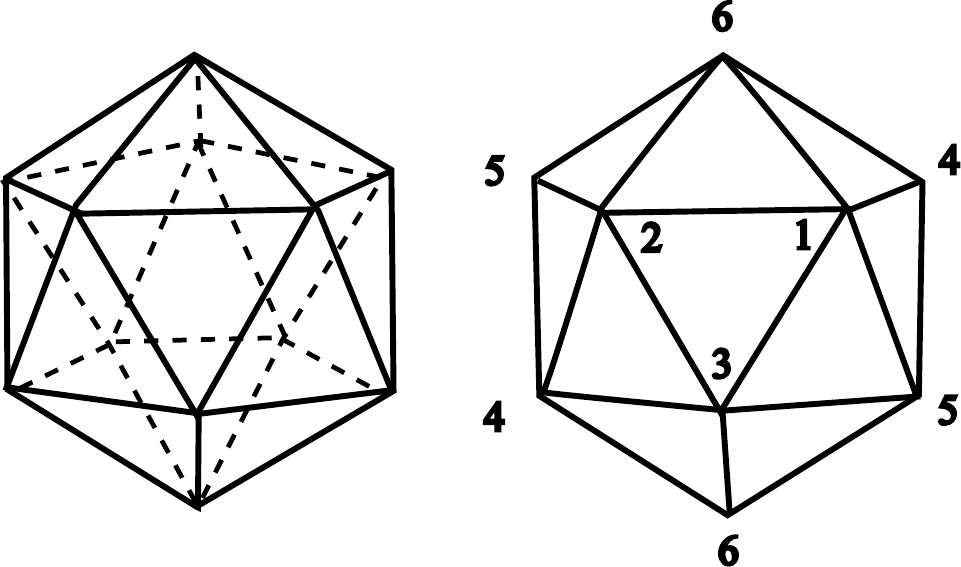}
\caption{{Hemi-icosahedron}}
\label{fig:hemi}
\end{figure}

\subsection{Polytopality of hemi-icosahedral Bier sphere}

\begin{theo}\label{thm:iko}
  The Bier sphere $Bier(\mathbb{I}_6)$ of the minimal, 6-vertex triangulation $\mathbb{I}_6$ of the real projective plane $\mathbb{R}P^2$ is polytopal, i.e.\ it can be realized as the boundary sphere of a five dimensional convex polytope. More explicitly, the vertices of the geometric realisation of $Bier(\mathbb{I}_6)$, obtained by the algorithm described in Section \ref{sec:Marinko}, are coordinatized in $\mathbb{R}^5$ as the rows of the matrix (\ref{matrix:hemi-icosahedron-vertices}).
\end{theo}


 \begin{equation}\label{matrix:hemi-icosahedron-vertices}
\left[
\begin{array}{ccccc}
 3.4083657 & 2.5824889 & 2.7636015 & 4.2308106 & 1.8312304\\
4.7537197 & 5.7195211 & 4.3242208 & 5.0848353 & 4.7010916\\
-3.6225148 & -3.6600717 & -2.256866 & -3.0205762 & -3.2792018\\
-0.5413665 & -1.3194438 & 0.1968267 & 2.6400609 & -2.4040568\\
1.200048 & 2.0558616 & 0.1089153 & -2.0769712 & 4.0182034\\
-3.8440022 & -3.7596347 & -4.3286914 & -6.2082317 & -2.9269678\\
-2.018964 & -0.6272048 & -0.8564941 & -2.3165986 & -0.1697324\\
-4.3439331 & -4.8032114 & -3.0078281 & -3.1808993 & -4.0042503\\
-1.5662252 & -1.3174413 & -2.9671459 & -4.6607583 & -0.1247618\\
0.1168746 & -0.1838719 & -0.1546464 & -1.6760483 & 0.2340882\\
1.232693 & 0.4628524 & 1.7251772 & 4.4773757 & -1.512764\\
5.2253047 & 4.8501556 & 4.4529303 & 6.7070012 & 3.6371211
\end{array}
\right]
\end{equation}

\medskip
\begin{proof} 
Let $Q$ be the convex hull of the row vectors of the matrix (\ref{matrix:hemi-icosahedron-vertices}).
Applying \emph{Polymake} \cite{polymake} to the data above and relabeling the vertices with
\begin{equation*}
\sigma : \{0, 1, 2, 3, 4, 5, 6, 7, 8, 9, 10, 11\}\longrightarrow \{1,2,3,4,5,6, \bar{1}, \bar{2}, \bar{3}, \bar{4}, \bar{5},\bar{6}\}
\end{equation*}
we obtain the face lattice of $Q$. Its facets and edges are given in tables (\ref{table:hemi-icosahedron-facets}) and (\ref{table:hemi-icosahedron-edges}).

\begin{equation}\label{table:hemi-icosahedron-facets}
Facets(Q)=\left\lbrace
\begin{array}{ccccc}
{4 6 \bar{1} \bar{2} \bar{3}} & {3 4 6 \bar{1} \bar{2}} & {1 6 \bar{2} \bar{3} \bar{4}} & {3 5 6 \bar{1} \bar{4}} & {1 5 \bar{2} \bar{3} \bar{4}} \\
{3 5 6 \bar{1} \bar{2}} & {3 6 \bar{1} \bar{4} \bar{5}} & {5 6 \bar{2} \bar{3} \bar{4}} & {3 4 6 \bar{1} \bar{5}} & {3 5 6 \bar{2} \bar{4}} \\
{4 5 \bar{1} \bar{2} \bar{3}} & {5 6 \bar{1} \bar{2} \bar{3}} & {3 4 6 \bar{2} \bar{5}} & {1 4 6 \bar{2} \bar{3}} & {3 6 \bar{2} \bar{4} \bar{5}} \\
{4 6 \bar{1} \bar{3} \bar{5}} & {1 4 6 \bar{3} \bar{5}} & {1 6 \bar{2} \bar{4} \bar{5}} & {1 3 \bar{2} \bar{4} \bar{5}} & {1 3 5 \bar{2} \bar{4}} \\
{1 4 5 \bar{2} \bar{3}} & {1 4 6 \bar{2} \bar{5}} & {2 3 4 \bar{1} \bar{5}} & {2 5 6 \bar{3} \bar{4}} & {2 6 \bar{1} \bar{4} \bar{5}} \\
{2 3 \bar{1} \bar{4} \bar{5}} & {2 5 6 \bar{1} \bar{4}} & {1 2 3 \bar{4} \bar{5}} & {1 2 6 \bar{3} \bar{4}} & {1 2 6 \bar{4} \bar{5}} \\
{1 2 6 \bar{3} \bar{5}} & {2 4 5 \bar{1} \bar{3}} & {2 6 \bar{1} \bar{3} \bar{5}} & {2 4 \bar{1} \bar{3} \bar{5}} & {2 5 6 \bar{1} \bar{3}} \\
{2 4 \bar{3} \bar{5} \bar{6}} & {2 4 5 \bar{3} \bar{6}} & {2 4 5 \bar{1} \bar{6}} & {1 2 \bar{3} \bar{5} \bar{6}} & {1 2 \bar{3} \bar{4} \bar{6}} \\
{1 2 3 \bar{4} \bar{6}} & {1 2 3 \bar{5} \bar{6}} & {2 5 \bar{1} \bar{4} \bar{6}} & {2 3 \bar{1} \bar{4} \bar{6}} & {2 5 \bar{3} \bar{4} \bar{6}} \\
{2 3 4 \bar{1} \bar{6}} & {2 3 4 \bar{5} \bar{6}} & {1 4 \bar{2} \bar{5} \bar{6}} & {3 4 \bar{1} \bar{2} \bar{6}} & {3 5 \bar{1} \bar{4} \bar{6}} \\
{1 5 \bar{3} \bar{4} \bar{6}} & {3 5 \bar{1} \bar{2} \bar{6}} & {4 5 \bar{1} \bar{2} \bar{6}} & {3 4 \bar{2} \bar{5} \bar{6}} & {1 4 \bar{3} \bar{5} \bar{6}} \\
{1 3 \bar{2} \bar{5} \bar{6}} & {1 3 5 \bar{4} \bar{6}} & {1 3 5 \bar{2} \bar{6}} & {1 4 5 \bar{3} \bar{6}} & {1 4 5 \bar{2} \bar{6}}
\end{array}
\right\rbrace
\end{equation}

\begin{equation}\label{table:hemi-icosahedron-edges}
Edges(Q)=\left\lbrace
\begin{array}{ccccc}
{1 2} & {1 3} & {1 4} & {1 5} & {1 6} \\
{1 \bar{2}} & {1 \bar{3}} & {1 \bar{4}} & {1 \bar{5}} & {1 \bar{6}} \\
{2 3} & {2 4} & {2 5} & {2 6} & {2 \bar{1}} \\
{2 \bar{3}} & {2 \bar{4}} & {2 \bar{5}} & {2 \bar{6}} & {3 4} \\
{3 5} & {3 6} & {3 \bar{1}} & {3 \bar{2}} & {3 \bar{4}} \\
{3 \bar{5}} & {3 \bar{6}} & {4 5} & {4 6} & {4 \bar{1}} \\
{4 \bar{2}} & {4 \bar{3}} & {4 \bar{5}} & {4 \bar{6}} & {5 6} \\
{5 \bar{1}} & {5 \bar{2}} & {5 \bar{3}} & {5 \bar{4}} & {5 \bar{6}} \\
{6 \bar{1}} & {6 \bar{2}} & {6 \bar{3}} & {6 \bar{4}} & {6 \bar{5}} \\
{\bar{1} \bar{2}} & {\bar{1} \bar{3}} & {\bar{1} \bar{4}} & {\bar{1} \bar{5}} & {\bar{1} \bar{6}} \\
{\bar{2} \bar{3}} & {\bar{2} \bar{4}} & {\bar{2} \bar{5}} & {\bar{2} \bar{6}} & {\bar{3} \bar{4}} \\
{\bar{3} \bar{5}} & {\bar{3} \bar{6}} & {\bar{4} \bar{5}} & {\bar{4} \bar{6}} & {\bar{5} \bar{6}}
\end{array}
\right\rbrace
\end{equation}
We read off the facets of $Bier(\mathbb{I}_6)$ from Figure \ref{fig:hemi}. It turns out that this is precisely the list exhibited in table (\ref{table:hemi-icosahedron-facets}). \end{proof}

The reader may wonder why the entries of the matrix (\ref{matrix:hemi-icosahedron-vertices}) are chosen to 7 decimal places and what happens if these numbers are rounded to, say, 5 decimal places. 

\medskip
Surprisingly enough 5 decimal places are not enough, the corresponding table (\ref{table:hemi-icosahedron-facets}) no longer represents the face lattice of $Bier(\mathbb{I}_6)$ and the convexity is lost.

\medskip
In reality the main algorithm \cite{TZJ24}  (as outlined in Section \ref{sec:Marinko}) works in floating-point arithmetic with 20 significant digits. It turned out that rounding to 7 decimal places was the best we can get, without loosing the convexity of the sphere $Bier(\mathbb{I}_6)$.

\section{Appendix: Bier spheres}\label{sec:app}

\subsection{Glossary}

\medskip
For the reader's convenience here is a glossary with brief descriptions of the main objects associated to  \emph{Bier spheres} \cite{Bier, Matousek}. For a more complete exposition of the geometry of Bier spheres see \cite{Zivaljevic19, jevtic_bier_2022, TZJ24}.

\medskip
$Bier(K) = K\ast_\Delta K^\circ$, the Bier sphere of $K$, see \cite{Matousek}, is a combinatorial object (simplicial complex), defined as a deleted join of two simplicial complexes, $K\subset 2^{[n]}$ and its Alexander dual $$K^\circ := \{A\subset [n] \mid A^c\notin K\}\, .$$

\medskip
$\mathcal{R}_{\pm\delta}(Bier(K))\subset H_0 \cong  \mathbb{R}^{n}/ \mathbb{R}\cong \mathbb{R}^{n-1}$
is the \emph{canonical starshaped realization} of $Bier(K)$ described in \cite[Theorem 3.1]{jtz-bier-2019}.

\medskip
$Star(K)$ is the star-shaped body whose boundary is the sphere $\mathcal{R}_{\pm\delta}(Bier(K))$.

\medskip
$Fan(K) = \mathcal{R}(Star(K))$, the \emph{canonical} or the \emph{Bier fan} of $K$, is the radial fan of the star-haped body $Star(K)$.

\medskip
$\Omega_n$ is a universal, $(n-1)$-dimensional convex polytope (the Van Kampen-Flores polytope) which is equal, as a convex body, to $Star(K)$ for each Bier sphere of maximal volume.

\subsection{Bier spheres of maximal volume}

\begin{prop}\label{prop:max-vol} {\rm (\cite[Proposition 6]{jevtic_bier_2022})}
 If $n=2k+1$ is odd the unique Bier sphere of maximal volume is $Bier(K)$ where
 \begin{equation}\label{eqn:VK-F-1}
K =\binom{[n]}{\leq k} = \{ S\subset [n] \mid \vert S\vert \leq k \} \, .
\end{equation}\label{eqn:even}
 If $n=2k$ is even a Bier sphere $Bier(K)$ is of maximal volume if and only if
 \begin{equation}\label{eqn:VK-F-2}
 \binom{[n]}{\leq k-1} \subseteq K \subseteq \binom{[n]}{\leq k} \, .
 \end{equation}
 A Bier sphere $Bier(K)$ is of minimal volume if and only if either $K = \{\emptyset\}$ or $K$ is the boundary of the simplex $\Delta_{[n]}$, $K = \partial\Delta_{[n]}= 2^{[n]}\setminus \{[n]\}$.
\end{prop}

\begin{cor}\label{cor:VK-F}{\rm (\cite[Corollary 7]{jevtic_bier_2022})}
For all Bier spheres $Bier(K)$ of maximal volume, the convex body $\Omega_n = Star(K)$ is unique and independent of $K$. The body $\Omega_n$ is centrally symmetric. More explicitly $\Omega_n = {\rm Conv}(\Delta_\delta \cup \nabla_\delta)$ where $\Delta_\delta \subset H_0$ is the simplex spanned by vertices $\delta_i := e_i - {\frac{1}{n}}(e_1+\dots+ e_n)$ and $\nabla_\delta := -\Delta_\delta = \Delta_{\bar\delta}$ is the simplex spanned by $\bar\delta_i = -\delta_i$.
\end{cor}

\subsection{Polytopal Bier spheres}

The following theorem was proved in \cite{Zivaljevic19}.

\begin{theo}\label{thm:Bier-polytope}
Bier sphere $Bier(T_{\mu_L<\nu})$ of a threshold complex 
$$
T_{\mu_L<\nu} := \{A\subset [n] \mid \mu_L(A) < \nu \} \subset 2^{[n]}
$$ 
is isomorphic to the boundary sphere of a convex polytope. More explicitly $Bier(T_{\mu_L<\nu})$
is isomorphic to the boundary sphere of the convex hull of the union of two simplices,
 \begin{equation}\label{eqn:conv-poly}
 Q_\alpha :=  {\rm Conv}(\Delta\cup\nabla_\alpha) = {\rm Conv}\{y_1,y_2,\dots, y_n, -\alpha y_1, -\alpha y_2, \dots, -\alpha y_n\}\, .
 \end{equation}
\end{theo}
Recall that the vertices of the simplex $\Delta := {\rm Conv}\{y_1,y_2,\dots, y_n\}$ are obtained as radial perturbation $y_i = {u_i}/{l_i}$ of the vertices of a regular simplex with vertices $u_i$ (centered at the origin), where $L = (l_1, l_2, \dots, l_n)$ is the vector of weights, associated to the probability measure $\mu_L$ and $\alpha = \frac{1-\nu}{\nu}$.

\medskip

\begin{cor}{\rm (\cite{jevtic_bier_2022}, Theorem 2.2)} $Bier(T_{\mu_L<\nu})$ is isomorphic to the boundary sphere of a convex polytope which can be realized as a polar dual of a generalized permutohedron.
\end{cor}


\end{document}